\documentclass{article}
\usepackage{amsmath,amssymb,amsthm}
\usepackage{url,graphicx}
\usepackage{todonotes}

\newtheorem{lem}{Lemma}

\newtheorem{thm}{Theorem}
\newtheorem{cor}{Corollary}

\newtheorem{example}{Example}
\newtheorem{rem}{Remark}

\renewcommand{\P}{\mathbb P}

\newcommand{\Z}{\mathbb Z}
\newcommand{\C}{\mathbb C}
\newcommand{\D}{\mathbb D}
\newcommand{\R}{\mathbb R}
\renewcommand{\H}{\mathbb H}
\newcommand{\eps}{\epsilon}

\DeclareMathOperator{\norm}{N}

\newcommand{\SO}[1][3]{\mathrm{SO}_{#1}}
\newcommand{\SE}[1][3]{\mathrm{SE}_{#1}}
\newcommand{\ci}{\mathrm{i}}
\newcommand{\qi}{\mathbf{i}}

\title{The Theory of Bonds II: \\ Closed 6R Linkages with Maximal Genus}

\author{G\'abor Heged\"us, Obuda University, Budapest, Hungary \\
	Zijia Li, RICAM, Linz, Austria \\
	Josef Schicho, RICAM, Linz, Austria \\
	Hans-Peter Schr\"ocker, University of Innsbruck, Austria }

\begin{document}

\maketitle

\begin{abstract}
In this paper, we introduce a method that allows to produce necessary conditions
on the Denavit--Hartenberg parameters for the mobility of a closed linkage
with six rotational joints. We use it to prove that the genus of the configuration curve of
a such a linkage is at most five, and to give a complete classification of the linkages 
with a configuration curve of degree four or five. The classification contains new families.
\end{abstract}

\section*{Introduction}

A linkage is a mechanism composed of a finite number of rigid bodies,
called links, and connections between them, called joints. The links
move in three-dimensional space, and when two links are connected by
a joint, then the relative motion is constrained to a certain subgroup
of the group of Euclidean displacements, depending on the joint.
For instance, a revolute joint ensures that the relative motion
is always a rotation around a fixed axis. A linkage consisting of
$n$ links that are cyclically connected by $n$ revolute joints
is called a closed $n$R linkage.

In kinematics, one studies the set of all possible configurations
of a linkage. If the configuration set has positive dimension,
then the linkage is mobile. This is always the case for $n$R linkages
when $n\ge 7$. There are mobile closed $n$R linkages for $n=4,5,6$.
A mobile closed 4R linkage is either planar, or spherical, or 
a Bennett linkage~\cite{bennett14}. For 5R linkages, we have
a similar classification that has been completed by \cite{Karger}.
The classification of mobile closed 6R linkages is still an open
problem.

The theory of bonds was introduced in \cite{hegedus13b} as a 
method for the analysis of linkages with revolute joints. The
configuration curve of such a linkage can be described by algebraic
equations. Intuitively, bonds are points in the configuration curve
with complex coefficients where something degenerate happens. For a
typical bond of a closed $n$R linkage, there are exactly two joints with
degenerate rotation angles.
In this way, the bond ``connects'' the two links. 

%

The theory of bonds has been used in \cite{hegedus13b} to give an almost computation-free
proof of the classification of closed 5R linkages. The original proof \cite{Karger}
is based on complex computations done with computer algebra. In \cite{nawratil}, bonds are
used for studying Stewart-Gough platforms with self-motions.

The main result in this paper is the classification of all 6R linkages with a configuration
curve of maximal genus. In Section~\ref{sec:bd}, we show that the maximum is 5 (examples of genus 5 are well-known,
but the fact that 5 is an upper bound is new). In Section~\ref{sec:max}, we give a classification of all
linkages with configuration curve of genus 5 in terms of their Denavit--Hartenberg parameters.
It turns out that they come in 4 families, two are well-known and two are new.

For the reader who is more interested in computation, the most interesting section is Section~\ref{sec:quad},
which is logically independent of Section~\ref{sec:bd}. Here we introduce a technique which allows
to produce polynomials in the Denavit--Hartenberg parameters such that their vanishing is necessary for the existence
of bonds.

It is apparent that the results of bond theory have their main interest in the field
of kinematics. However, we also want to address algebraists and geometers, because
we hope to serve as an inspiration to use this technique and develop similar ones 
for solving more questions of interest in kinematics.

\section{Preliminary Definitions}
\label{sec:linkages}

In this section we recall several classical concepts and definitions that we need later:
linkages and their configuration set and coupler maps, the Study quadric, and dual quaternions.

We denote by $\SE$ the group of Euclidean displacements, i.e.,
the group of maps from $\R^3$ to itself that preserve distances and
orientation. It is well-known that $\SE$ is a semidirect product of
the translation subgroup and the orthogonal group $\SO$, which may be
identified with the stabilizer of a single point.

We denote by $\D:=\R+\eps\R$ the ring of dual numbers, with
multiplication defined by $\eps^2=0$. The algebra $\H$ is the
non-commutative algebra of quaternions, and $\D\H:=\D\otimes_\R\H$. 
The conjugate dual quaternion
$\overline{h}$ of $h$ is obtained by multiplying the vectorial part of $h$
by $-1$. The dual numbers $\norm(h) = h\overline{h}$ and $h+\overline{h}$ are
called the \emph{norm} and \emph{trace} of $h$, respectively.

By projectivizing $\D\H$ as a real 8-dimensional vector space, we
obtain $\P^7$. The condition that $\norm(h)$ is strictly real, i.e.\
its dual part is zero, is a homogeneous quadratic equation. Its zero
set, denoted by $S$, is called the Study quadric. The linear 3-space
represented by all dual quaternions with zero primal part is denoted
by $E$. It is contained in the Study quadric. The complement $S-E$ is
closed under multiplication and multiplicative inverse and therefore
forms a group, which is isomorphic to $\SE$ (see \cite[Section~2.4]{husty10}).

A nonzero dual quaternion represents a rotation if and only if its
norm and trace are strictly real and its primal vectorial part is
nonzero. It represents a translation if and only if its norm and trace
are strictly real and its primal vectorial part is zero. The
1-parameter rotation subgroups with fixed axis and the 1-parameter
translation subgroups with fixed direction can be geometrically
characterized as the lines on $S$ through the identity element $1$.
Among them, translations are those lines that meet the exceptional
3-plane~$E$.

Let $n\ge 4$. For the analysis of the configurations of a closed $n$R linkage with
links $o_1,\dots,o_n$, the actual shape of links is irrelevant; it is enough to know
the position of the rotation axes. Exploiting the fact that there is a bijection between lines
in $\R^3$ and involutions in $\SE$, we describe a closed $n$R linkage by a sequence $L
= (h_1,\ldots,h_n)$ of dual quaternions $h_1,\ldots,h_n$ such that $h_i^2=-1$ and $h_i\ne\pm h_{i+1}$
for $i=1,\dots,n$
(we set $h_{i+kn}=h_i$ and $o_{i+kn}=o_i$ for all $k\in\Z$). 
The line $h_i$ specifies the joint connecting the links $o_{i-1}$ and $o_i$.
The subgroup of rotations with axis $h_i$ is parametrized by $(t-h_i)_{t\in\P^1}$.
The pose of $o_i$ with respect to $o_n$ is 
then given by a product $(t_1-h_1)(t_2-h_2)\cdots(t_i-h_i)$, with
$t_1,\dots,t_i\in\P^1$. Setting $i:=n$, we get the closure condition
\begin{equation}
  \label{eq:1}
  (t_1-h_1)(t_2-h_2)\cdots(t_n-h_n) \in \R^\ast.
\end{equation}
The set $K$ of all $n$-tuples $(t_1,\dots,t_n)\in(\P^1)^n$ fulfilling
\eqref{eq:1} is called the \emph{configuration set} of the linkage $L$.

The dimension of the configuration set is called the \emph{mobility} of the linkage. 
We are mostly interested in linkages of mobility one.
Let $L=(h_1,\dots,h_n)$ be such a linkage. Let $K$ be its configuration curve.
For any two pair $o_i,o_j$ of links, there is a map
\[ f_{i,j}: K\to\P^7, (t_1,\dots,t_n)\to (t_{i+1}-h_{i+1})\dots(t_j-h_j) \]
parametrizing the motion of $o_j$ with respect to $o_i$. This map is
is called coupler map, and the image $C_{i,j}$ is the coupler curve. 
The {\em algebraic degree} of the coupler curve is defined as $\deg(C_{i,j})\deg(f_{i,j})$,
where $\deg(C_{i,j})$ is the degree of $C_{i,j}$ as a projective curve, and $\deg(f_{i,j})$
is the degree of $f_{i,j}$ as a rational map $K\to C_{i,j}$.

\section{Bonds: Definition and Main Properties}

In this section we recall the fundamentals of bond theory, as introduced in \cite{hegedus13b}.

Let $n\ge 4$ be an integer. Let $L=(h_1,\dots,h_n)$ be a closed $n$R linkage with mobility~1.
We assume, for simplicity, that the configuration curve $K\subset(\P^1_\R)^n$ has only one component
of dimension~1 (see Remark~\ref{rem:reduc} for a comment on the reducible case).
Let $K_\C\subset(\P^1_\C)^n$ be the Zariski closure of $K$. We set
\begin{equation}
  \label{eq:6}
  B := \{(t_1,\ldots,t_n) \in K_\C \mid
  (t_1-h_1)(t_2-h_2)\cdots(t_n-h_n) = 0\}.
\end{equation}
The set $B$ is a finite set of conjugate complex points on the
configuration curve's Zariski closure. 
If $K$ is a nonsingular curve,  then we define a bond as a point of $B$. 
If $K$ has singularities, then it is necessary to pass to the normalization $N(K)$
of $K$ as a complex algebraic curve, and a bond is then a point on $N(K)$ lying over $B$.
Zero-dimensional components of $K$ never fulfill the equation above and so they have no
effect on bonds.

Let $\beta$ be a bond lying over $(t_1,\ldots,t_n)$. By Theorem~2 in \cite{hegedus13b},
there exist indices $i,j \in [n]$, $i < j$, such that $t_i^2 + 1 = t_j^2 + 1 = 0$. 
If there are exactly two coordinates of $\beta$ with values $\pm \ci$, then
we say that $\beta$ {\em connects} joints $i$ and $j$. In general, the
situation, is more complicated. Let $\beta\in N(K)$ be a bond; we assume, for simplicity,
that it lies over a point $(t_1,\ldots,t_n)$ such that no $t_i$ is the infinite point
in $\P^1$.
For $i,j\in\{1,\dots,n\}$, we define 
\begin{equation}
  \label{eq:2}
  \begin{gathered}
    F_{i,j}(\beta) = (t_{i+1}(\beta)-h_{i+1})\cdots(t_j(\beta)-h_j)
    \in \D\H,
  \end{gathered}
\end{equation}
The distinction between $F_{i,j}$ and $f_{i,j}$
is necessary because $F_{i,j}$ may vanish at the bonds, and then it does not
give a well-defined pose in $\P^7$. We define $v_\tau(i,j)$ as the vanishing
order of $F_{i,j}$ at $\tau$. 
We define the connection number
\[ k_\beta(i,j) := v_\beta(i,j-1)+v_\beta(i-1,j)-v_\beta(i,j)-v_\beta(i-1,j-1) . \]

We visualize bonds and their connection numbers by \emph{bond diagrams.}
We start with the link diagram, where vertices correspond to links and edges
correspond to joints.
Then we draw $k_\beta(i,j)$ connecting lines between
the edges $h_i$ and $h_j$ for each set $\{\beta, \overline{\beta}\}$ of
conjugate complex bonds.
Since we cannot exclude that $k_\beta(i,j) < 0$, we visualize
negative connection numbers by drawing
the appropriate number of dashed connecting lines (because the
dash resembles a ``minus'' sign). No linkage in this paper
has a negative connection number. Actually, the authors do not know
if closed 6R linkages may or may not have bonds with negative
connection numbers.

\begin{thm}
  \label{thm:degree}
  The algebraic degree of the coupler curve $C_{i,j}$ can be read off
  from the bond diagram as follows: Cut the bond diagram at the
  vertices $o_i$ and $o_j$ to obtain two chains with endpoints
  $o_i$ and $o_j$; the algebraic degree of $C_{i,j}$ is the sum of all
  connections that are drawn between these two components
  (dashed connections counted negatively).
\end{thm}

\begin{proof}
This is a consequence of Theorem 5 in \cite{hegedus13b}. Note that here
we give a different definition of connection numbers, but Lemma~2 in \cite{hegedus13b}
shows that the definitions are equivalent.

The basic idea of the proof is that the algebraic degree of $C_{i,j}$ is $\frac{1}{2}$ times
the number of points $\tau$ in the configuration curve such that $\norm(f_{i,j}(\tau))=0$.
All these points are bonds, and a closer investigation leads to the statement above.
\end{proof}

\begin{example}
  \label{ex:10}
  We illustrate the procedure for computing the degrees
  in Figure~\ref{fig:degree}. In order to determine the
  algebraic degree of the coupler curve $C_{3,5}$, we cut the bond diagram along
  the line through $o_3$ and $o_5$ and count the connections between
  the two chain graphs. There are precisely two of them, one
  connecting $h_1$ with $h_4$ and one connecting $h_2$ with $h_5$.
  Thus, the algebraic degree $d(3,5)$ of $C_{3,5}$ is two. 
\end{example}

\begin{figure}[htb]
  \centering
  \includegraphics{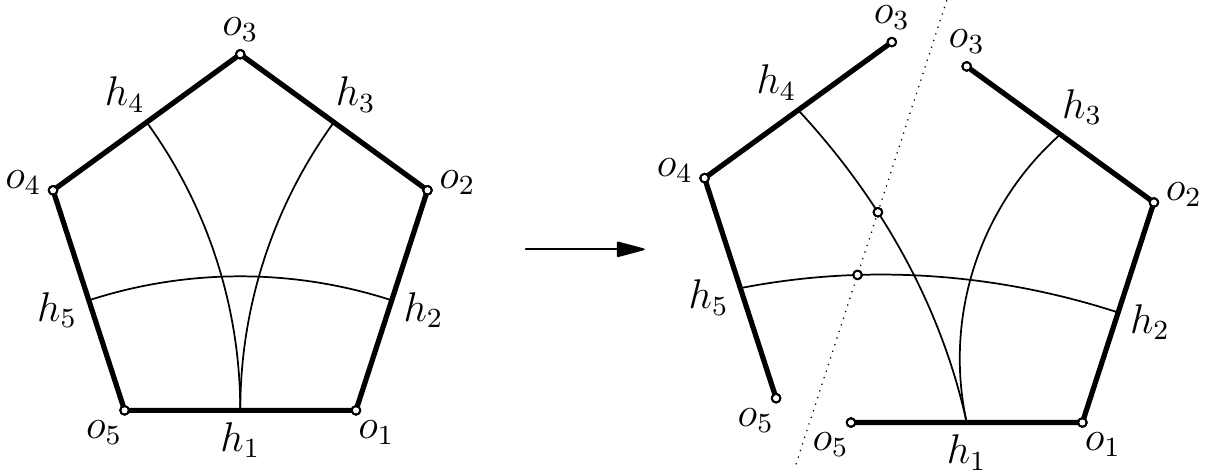}
  \caption{Computing the degree of coupler curves by counting
    connections in the bond-diagram. There are two connections 
    between the two chains, hence the algebraic degree of the
    coupler curve $C_{3,5}$ is two. }
  \label{fig:degree}
\end{figure}

For a sequence $h_i,h_{i+1},\dots,h_j$ of consecutive joints, we
define the \emph{coupling space} $L_{i,i+1,\dots,j}$ as the linear
subspace of $\R^8$ generated by all products $h_{k_1}\cdots
h_{k_s}$, $i\le k_1<\cdots <k_s\le j$. (Here, we view dual
quaternions as real vectors of dimension eight.) The empty product
is included, its value is $1$. The \emph{coupling dimension}
$l_{i,i+1,\dots,j}$ is the dimension of $L_{i,i+1,\dots,j}$ and the
\emph{coupling variety $X_{i,i+1,\dots,j}\subset\P^7$} is the set of
all products $(t_i-h_i)\cdots(t_j-h_j)$ with $t_k\in\P^1$ for
$k=i,\dots,j$ or, more precisely, the set of all equivalence classes
of these products in the projective space.

The coupling variety is a subset of the projectivization of the
coupling space. The relation between the coupler curve and the
coupling variety is described by the ``coupler equality''
$C_{i,j}=X_{i+1,\dots,j}\cap X_{i,\dots,-n+j+1}$.

The relation between bonds and coupling dimensions is described in the following

\begin{thm} \label{thm:coup}
  All coupling dimensions $l_{1,\ldots,i}$ with $1 \le i \le n$ are even. 
  We have $l_{1,2}=4$ and $k_\beta(1,2)=0$ for every bond $\beta$.
  If $k_\beta(1,3)\ne 0$ for some $\beta$, then $l_{1,2,3}\le6$.
  If $l_{1,2,3}=4$, then the lines $h_1,h_2,h_3$ are parallel or have a common point. 
\end{thm}

\begin{proof}
This is part of Theorem~1, Theorem~3, and Corollary~3 in \cite{hegedus13b}. 
The first statement is a consequence of the fact that the coupling spaces 
can be given the structure of a complex vector space,
because they are closed under multiplication by $h_1$ from the left.
\end{proof}

We will also use a more precise description of the coupling varieties in each
of the three possible cases, which is interesting in itself.

\begin{thm} \label{thm:cvar}
If $l_{1,2,3}=4$, then $X_{1,2,3}$ is a linear projective 3-space, and its parametrization
by $t_1,t_2,t_3$ is a 2:1 map branched along two quadrics in this 3-space.

If $l_{1,2,3}=6$, then $X_{1,2,3}$ is either a complete intersection of two quadrics or
an intersection of three quadrics in a 5-space
and its parametrization by $t_1,t_2,t_3$ is birational. 

If $l_{1,2,3}=8$, then $X_{1,2,3}$ is a Segre embedding of $(\P^1)^3$ in $\P^7$, and
its parametrization by $t_1,t_2,t_3$ is an isomorphism.
\end{thm}

\begin{proof}
The first statement is well-known in kinematics. For non-parallel axes
it is, for example, implicit in the exposition of
\cite[Section~5]{selig05}. Branching occurs for co-planar joint axes,
that is, for two fixed values of $t_2 = u$ and $t_2 = -u^{-1}$. The
remaining 2R chains generate the two quadrics.

If $l_{1,2,3}=6$, then $(\ci-h_1)(s_2-h_2)(\pm\ci-h_3)=0$ for some $s_2\in\P^1_\C$, by 
the proof of Theorem~1 in \cite{hegedus13b}; we may assume that the third factor is $(+\ci-h_3)$.
Clearly there is also a complex conjugate relation $(-\ci-h_1)(\bar{s_2}-h_2)(-\ci-h_3)=0$.
The parametrization $p:(\P^1)^3\to X_{1,2,3}$ has two base points $(\ci,s_2,\ci)$ and $(-\ci,\overline{s_2},-\ci)$.
We distinguish two cases.

If $s_2\ne\overline{s_2}$, then we apply projective transformations moving the base points to $(0,0,0)$
and $(\infty,\infty,\infty)$. The transformed parametrization is
\[ (\P^1)^3\to\P^5, (y_1,y_2,y_3) \mapsto (x_0{:}x_1{:}x_2{:}x_3{:}x_4{:}x_5)
	=(y_1{:}y_2{:}y_3{:}y_1y_2{:}y_1y_3{:}y_2y_3) , \]
which is birational to the quartic three-fold defined by $x_0x_5=x_1x_4=x_2x_3$.

If  $s_2=\overline{s_2}$, then we apply projective transformations moving the base points to $(0,\infty,0)$ 
and $(\infty,\infty,\infty)$. The transformed parametrization is
\[ (\P^1)^3\to\P^5, (y_1,y_2,y_3) \mapsto (x_0{:}x_1{:}x_2{:}x_3{:}x_4{:}x_5)=(1{:}y_1{:}y_3{:}y_2{:}y_1y_2{:}y_2y_3) , \]
which is birational to the cubic three-fold defined by $x_0x_4=x_1x_3$, $x_0x_5=x_2x_3$, $x_1x_5=x_2x_4$.

If $l_{1,2,3}=8$, then the eight products generating $L_{1,2,3}$ are linearly independent, and it follows that the
parametrization is the Segre embedding in the projective coordinate system induced by this basis.
\end{proof}

All bonds connecting $h_1$ and $h_4$ satisfy $t_1,t_4\in\{+\ci,-\ci\}$. 
We will prove a lemma that is useful to give an upper bound for the number of bonds in some situations;
before that, we need an algebraic lemma.

\begin{lem} \label{lem:alg}
Let $h_1,h_2\in\D\H$ be dual quaternions representing lines (i.e. $h_1^2=h_2^2=-1$).
Let $\D_\C:=\D\otimes_\R\C$ and $\D\H_\C:=\D\H\otimes_\R\C$ be the extensions of the dual numbers/quaternions to $\C$.

(a) The left annihilator of $(\ci-h_1)$ is equal to the left ideal $\D\H_\C(\ci+h_1)$. 

(b) The intersection of this left ideal and the right ideal $(\ci-h_2)\D\H_\C$ is a free $\D_\C$-module
of rank~1. 

(c) The set of all complex dual quaternions $x$ such that $(\ci-h_2)x(\ci-h_1)=0$ is a free $\D_\C$-module
of rank~3.
\end{lem}

\begin{proof}
For $h_1=h_2=\qi$, the proofs for all three statements are straightforward.

The group of unit dual quaternions acts transitively on lines by conjugation, so there exist invertible
$g_1,g_2\in\D\H$ such that $h_1=g_1\qi g_1^{-1}$ and $h_2=g_2\qi g_2^{-1}$. Then
\[ \{ q\mid q(\ci-h_1)=0\} = \{ q\mid qg_1(\ci-\qi)g_1^{-1}=0\} =  \{ q\mid qg_1(\ci-\qi)=0 \} = \]
\[ \D\H_\C(\ci+\qi)g_1^{-1}= \D\H_\C g_1^{-1}(\ci+h_1)=\D\H_\C(\ci+h_1) , \]
which shows (a). The $\D_\C$-linear bijective map $\D\H_\C\to\D\H_\C$, $q\mapsto g_2^{-1}qg_1$ 
maps the left ideal $\D\H_\C(\ci+h_1)$ to the left ideal $\D\H_\C(\ci+\qi)$ and the right ideal $(\ci-h_2)\D\H_\C$ 
to the right ideal $(\ci-\qi)\D\H_\C$, which shows (b). The same map also maps the set $\{ x\mid (\ci-h_2)x(\ci-h_1)=0\}$ to the 
set $\{ x\mid (\ci-\qi)x(\ci-\qi)=0\}$, which shows (c).
\end{proof}

\begin{lem} \label{lem:bb}
Assume that $l_{1,2,3}=l_{4,5,6}=8$. Then there are at most 2 bonds $\beta:=(t_1,\dots,t_6)$ 
connecting $h_1,h_4$ for fixed values of $t_1$ and $t_4$ in $\{+\ci,-\ci\}$ (counted with multiplicity).
\end{lem}

\begin{proof}
Without loss of generality, we may assume $t_1=t_4=+\ci$; the other situations can reduced to this
case by replacing $h_1$ or $h_4$ or both by its negative. 

By the algebraic lemma above, the intersection of the left annihilator of $(t_4-h_4)$ 
and the right ideal $(t_1-h_1)\D\H_C$ is
a 2-dimensional $\C$-linear subspace. Let $G\subset\P^7$ be its projectivization.
Let $q:=f_{3,6}(\beta)$ be image of a bond $\beta=(t_1,\dots,t_6)$ connecting $h_1$ and $h_4$ with $t_1=t_4=+\ci$.
Then we have $q=(t_1-h_1)(t_2-h_2)(t_3-h_3)$, hence $q$ is in the right ideal $(t_1-h_1)\D\H_\C$.
Since $\beta$ connects $h_1$ and $h_4$, we have $q(t_4-h_4)=F_{3,6}(\beta)=0$,
$q$ is in the left annihilator of $(t_4-h_4)$, and therefore $q\in G$.

There exist  no two bonds $\beta_1,\beta_2$ with the same bond image $q$, because the parameterization
of $X_{1,2,3}$ by $t_1,t_2,t_3$ is an isomorphism, hence $q$ determines the first three coordinates,
and the parametrization of $X_{6,5,4}$ by $t_4,t_5,t_6$ is also an isomorphism, hence $q$ determines the second 
three coordinates.
This shows that the number of bonds connecting $h_1$ and $h_4$ with $t_1=t_4=+\ci$ is equal
to the number of intersections of $G$ and $C_{3,6}$; tangential intersections give rise to higher connection numbers.

On the other hand, $C_{3,6}$ is generated by quadrics, so
it does not have any tritangents, so the number of such bonds is at most 2.
\end{proof}

\section{Bounding the Genus} \label{sec:bd}

In this section, we prove that the genus of the configuration curve of a closed 6R linkage is at most 5.

Let $L=(h_1,\dots,h_6)$ be a closed 6R linkage with mobility~1.
We use the notation of the previous section.
As before, we assume that the configuration curve $K$ has only one irreducible one-dimensional component.
We write $g(K)$ for the genus of this component.

Here is an auxiliary Lemma.

\begin{lem} \label{lem:ag}
Let $C_1,C_2$ be two curves of genus at most 1. Let $C\subset C_1\times C_2$ be an
irreducible curve such that the two projections restricted to $C$ are either  birational
or 2:1 maps to $C_1$ resp. $C_2$. Then $g(C)\le 5$, with equality only if $g(C_1)=g(C_2)=1$
and both projections being 2:1.
\end{lem}

\begin{proof}
If one of the two projections is birational, say the projection to $C_1$, then $g(C)=g(C_1)\le 1$.
So we may assume both projections are 2:1 maps.

If $C_1$ and $C_2$ are isomorphic to $\P^1$, then $C$ is a curve in $\P^1\times\P^1$ of bi-degree~2,
which has arithmetic genus~1. The geometric genus is 1 in the nonsingular case and 0 if $C$ has
a double point.

If $C_1=\P^1$ and $C_2$ is elliptic, then the numerical class group is generated by the two fibers
$F_1\cong C_2$ and $F_2\cong C_1$ of the two projections. The class of $C$ is $2F_1+2F_2$, and
the canonical class is $-2F_2$. Hence the arithmetic genus of $C$ is
$\frac{C(C+K)}{2}+1=2(F_1+F_2)F_1+1=3$.

If $C_1$ and $C_2$ are elliptic, then the canonical class of $C_1\times C_2$ is zero. If $F_1,F_2$
are fibers of the projections, then $F_1C=F_2C=2$ and $(F_1+F_2)^2=2$. By the Hodge index theorem,
$(C-2F_1-2F_2)^2\le 0$, which is equivalent to $C^2\le 8$.

Hence the arithmetic genus of $C$
is at most $\frac{C^2}{2}+1=5$.
\end{proof}


\begin{lem} \label{lem:4}
If $l_{1,2,3}=4$, then $g(K)\le 5$. 
\end{lem}

\begin{proof}
Let $C_1,C_2\subset (\P^1)^3$ be the projections of $K$ to $(t_1,t_2,t_3)$ and $(t_4,t_5,t_6)$,
respectively. Let $p_1:K\to C_1$ and $p_2:K\to C_2$ be the projection maps. The coupler curve
$C_{3,6}$ is a common image of $C_1$ and $C_2$, by the two sides of
the closure equation
\[ (t_1-h_1)(t_2-h_2)(t_3-h_3) \equiv (t_6-h_6)(t_5-h_5)(t_4-h_4) , \]
where we write $\equiv$ for equality in the projective sense, modulo scalar multiplication.
Let $f_1:C_1\to C_{3,6}$ and $f_2:C_2\to C_{3,6}$ be these two maps. Then $K$ is a component of
the pullback of $f_1,f_2$. We distinguish several cases.

Case 1: $l_{6,5,4}=4$. Then $C_{3,6}$ is the intersection of two
linear subspaces, hence a line by the mobility 1 assumption. One can
introduce an additional joint, rotational or translational, between
links $o_3$ and $o_6$, and the linkage decomposes into two 4-bar
linkages which are planar or spherical. The configuration curves of
these two linkages are isomorphic to $C_1$ and $C_2$. The maps
$f_1,f_2$ are restrictions of the 2:1 parametrizations of $X_{1,2,3}$
and $X_{6,5,4}$, hence they are either 2:1 or birational to the line
$C_{3,6}$. Therefore $p_1$ and $p_2$ are also either 2:1 or
birational. The configuration curve of a planar or spherical 4-bar
linkage is the intersection curve of two quadrics (see
\cite[Chapter~11, \S~8]{bottema90} for the planar and
\cite[\S~21]{mueller62} for the spherical case). Hence its genus is at
most~1. By Lemma~\ref{lem:ag}, $g(K)\le 5$.

Case 2: $l_{6,5,4}=6$. Then $X_{6,5,4}$ is an intersection of quadrics in a 5-space
and $X_{1,2,3}$ is a linear 3-space contained in the Study quadric. The intersection of both linear
spaces is either a line or a plane, because the vector space $L_{6,5,4}$ does not contain any 4-dimensional
subalgebras. Hence the intersection $C_{3,6}$ is either a line or a plane conic. If $C_{3,6}$
is a line, then we have a similar situation as before: the linkage decomposes into two 4-bar
linkages, one planar or spherical and the second being a Bennett linkage. In any Bennett linkage,
the maps from $K$ to $\P^1$ parametrizing the 4 rotations are isomorphisms. Therefore $K$ is
isomorphic to the configuration curve of the planar or spherical component, hence $g(K)\le 1$.
If $C_{3,6}$ is a plane conic, then we decompose it into two rotational linear motions with
coplanar axes. These two  axes form with $h_3,h_4,h_5$ a closed 5R linkage, which is known
as the Goldberg 5R linkage (see \cite{Goldberg}). Its configuration curve is rational,
more precisely the coupling map to the plane conic is an isomorphism (see \cite{hegedus13b}). Hence $K$ is isomorphic
to $C_1$. Now $f_1:C_1\to C_{3,6}$ has 8 branching points (counted with multiplicity), namely
the intersections of $C_{3,6}$ with the branching surface. By the Hurwitz genus formula,
it follows that $g(K)=3$; the genus may drop in case of singularities.

Case 3: $l_{6,5,4}=8$. Then $C_{3,6}$ is a curve in a Segre embedding of $(\P^1)^3$ in $\P^7$
cut out by four hyperplane sections. This is only possible if $C_{3,6}$ is a twisted cubic.
Then the lines $h_4$, $h_5$, $h_6$ could be re-covered from $C_{3,6}$ by factoring the
cubic motion parametrized by $C_{3,6}$  described in \cite{hegedus13a}. On the other hand,
$C_{3,6}$ is either a planar or spherical motion, hence the whole linkage is either planar or spherical,
and this contradicts the mobility 1 assumption, as planar and spherical 6R linkages have
mobility 3. So this case is impossible.
\end{proof}

\begin{rem} \label{rem:4}
If $g(K)\ge 4$, then we are in Case~1, and the linkage is a composite of two planar or
spherical 4-bar linkages with one common joint, which is removed from the 6-loop. The most
general linkage of this type is Hooke's linkage \cite{Hooke}, using two spherical linkages.
The genus of its configuration curve is generically~5, but it may drop in the presence
of singularities.
If we take two planar RRRP linkages and remove the common translational joint, then we obtain
the Sarrus linkage \cite{sarrus53} with two triples of parallel
consecutive axes. The bond diagrams of  both linkages can be seen
in Figure~\ref{fig:hookediet}(a).
\end{rem}

\begin{lem} \label{lem:6}
If $l_{1,2,3}=l_{6,5,4}=6$, then $g(K)\le 5$.
\end{lem}

\begin{proof}
Let $V:=L_{1,2,3}\cap L_{6,5,4}$. Then $4\le \dim(V)\le 5$. The  $\dim(V)=6$ case is not possible by
Lemma~6 in \cite{hegedus13b}. If $\dim(V)=4$, then $C_{6,3}$ is embedded into a three dimensional projective space $\P^3$.
The coupler varieties are defined by quadrics in $\P^5$, 
therefore the ideal of $C_{6,3}$ is generated by linear forms and quadrics, and so its genus is at most 1. 
The coupler map $f_{6,3}$ is birational, therefore $g(K)\le 1$. 
So we may assume $\dim(V)=5$. 

By Theorem~\ref{thm:cvar}, the varieties $X_{1,2,3}$ and $X_{6,5,4}$ are either complete intersections
of two quadrics or intersections of three quadrics. We may assume in each case that one of the defining equations 
is the equation of the Study quadric. If both varieties are complete intersections, then the coupler curve 
$C_{6,3}=X_{1,2,3}\cap X_{6,5,4}$ is defined by three quadratic equations and the linear forms defining $V$. 
It follows that $C_{6,3}$ is a complete intersection of three quadrics in $\P^4$, which implies $g(K)\le 5$, with
equality in the case that there are no singularities. If one or both of the two coupling varieties is defined
by three quadrics, then $C_{3,6}$ is a component of the intersection of the Study quadric and two other quadrics,
where we choose one quadric from the ideal of $X_{1,2,3}$ and the other from the ideal of $X_{6,5,4}$; both
choices should be generic in order to not increase the dimension of the intersection. Then $C_{3,6}$ is a component
of a connected reducible curve of arithmetic genus 5, and this implies that the arithmetic genus of $C_{3,6}$ is also
at most 5.
\end{proof}

\begin{rem} \label{rem:6}
In \cite{Dietmaier}, Dietmaier found a new linkage by a computer-supported numerical search.
It turns out, by comparing the geometric parameters, that his family is exactly 
the family of linkages with $l_{1,2,3}=l_{4,5,6}=6$ and $\dim(L_{1,2,3}\cap L_{6,5,4})=5$.
See Figure~\ref{fig:hookediet}(b) for the bond diagram of the Dietmaier linkage.
\end{rem}

\begin{figure*}[htb]
  \centering
  \includegraphics{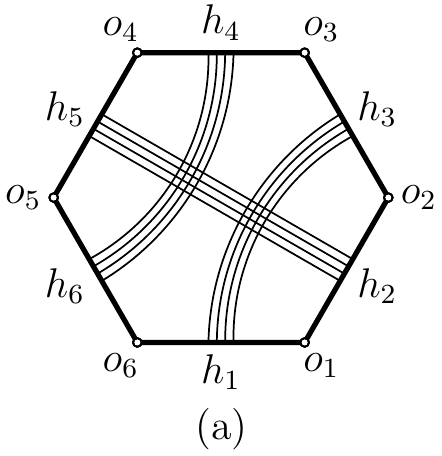}
  \quad
  \includegraphics{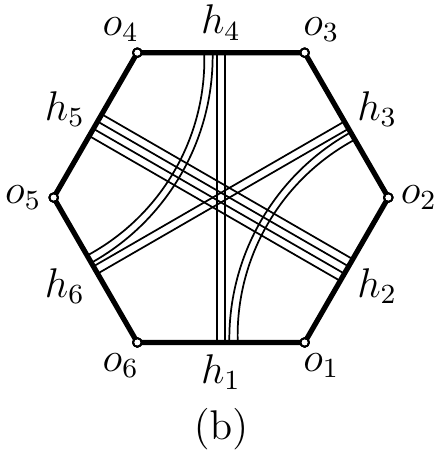}
  \quad
  \includegraphics{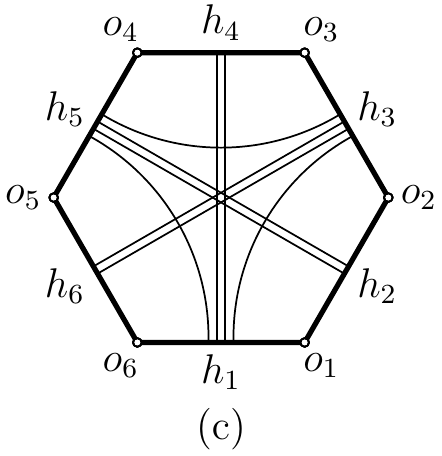}
  \quad
  \includegraphics{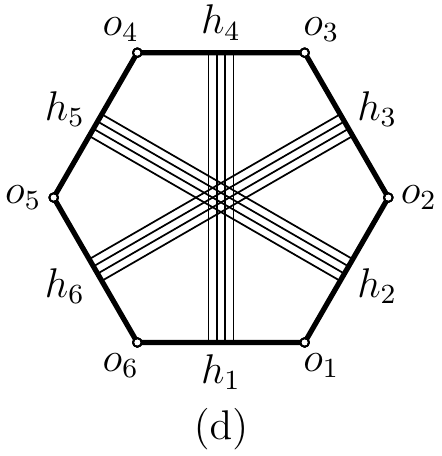}
  \caption{Bond diagrams for Hooke's double spherical linkage (a), Dietmaier's linkage (b), 
	Wohlhart's partially symmetric linkage (c), and Bricard's orthogonal linkage (d).}
  \label{fig:hookediet}
\end{figure*}

\begin{lem} \label{lem:68}
If $l_{1,2,3}=6$ and $l_{6,5,4}=8$, then $g(K)\le 3$.
\end{lem}

\begin{proof}
If $Y:=X_{6,5,4}\cap L_{1,2,3}$ has dimension~1, then its Betti table coincides with the Betti table
of $X_{6,5,4}$ and it follows that $Y$ is a union of curves with genus at most~1 (the genus 1 case  occurs
only if $Y$ is irreducible). Since $C_{6,3}\subseteq Y$, it follows that $g(C_{6,3})\le 1$, and
by birationality of $f_{6,3}$ we get $g(K)\le 1$.

Assume $Y$ is a surface. The preimage $Z$ of $Y$ under the parametrization $p:(\P^1)^3\to X_{6,5,4}$
is defined by two equations of tri-degree $(1,1,1)$, and because $Y$ is a surface, the two
equations must have a common divisor $F$ which defines $Z$. Up to permutation of coordinates,
the tri-degree of $F$ is either $(1,0,0)$ or $(1,1,0)$. In the first case, one of the angles
would be constant throughout the motion. Hence the 6R linkage is actually a 5R linkage with
an extra immobile axis somewhere; then $g(K)=0$ by the classification of 5R linkages (see \cite{hegedus13b} ) (if one
does not want to exclude this degenerate case). In the second case, we consider the
preimage $C'$ of $C_{6,3}$ under $p$. It is defined by $F$ and the pullback of the quadric
equations which defines $X_{1,2,3}$. Hence $C'$ is a component of the
complete intersection of two equations, with tri-degree $(1,1,0)$ and $(2,2,2)$. Using
the first equation, we can express the first variable by the second, and so we get an
isomorphic image of $C'$ in $(\P^1)^2$ of bi-degree $(4,2)$, which has arithmetic genus 3.
But $p$ is is an isomorphism by Theorem~\ref{thm:cvar}, hence $g(C_{6,3})\le 3$ and $g(K)\le 3$.
\end{proof}

\begin{rem}
An example of a linkage where $Y$ is a surface is
Wohlhart's partially symmetric linkage \cite{Wohlhart}
(see Figure~\ref{fig:hookediet}(c) for the bond diagram). 
We do not know if there exist also other linkages with $l_{1,2,3}=6$ and $l_{6,5,4}=8$
and $g(K)=3$.
\end{rem}

\begin{lem} \label{lem:8}
If $l_{i,i+1,i+2}=8$ for $i=1,\dots,6$, then $g(K)\le 5$.
\end{lem}

\begin{proof}
By Theorem~\ref{thm:coup}, all bonds connect opposite joints: the bond diagram consists of
$b_1$ connections between $h_1$ and $h_4$,
$b_2$ connections between $h_2$ and $h_5$, and
$b_3$ connections between $h_3$ and $h_6$.
By Theorem~\ref{thm:degree}, the degree of $f_{6,1}$ and the degree of $f_{3,4}$ are both equal to $b_1$.
Note that $f_{6,1}$ and $f_{3,4}$ are the projections from $K$ to the first and to the fourth coordinate, 
respectively, up to isomorphic parameterization of the line describing rotations around $h_1$ and $h_4$,
respectively.
Similarily, the projections to $t_2,t_5,t_3,t_6$ have degree $b_2,b_2,b_3,b_3$.

Let $b_1^+$ be the number of pairs of complex conjugate bonds connecting $h_1$ and $h_4$ such that $t_1=t_4$
and $b_1^-$ be the number of pairs such that $t_1=-t_4$ (recall that $t_1^2=t_4^2=-1$).
The numbers $b_2^+,b_2^-,b_3^+,b_3^-$ are defined analogously.
By Lemma~\ref{lem:bb}, we have $b_1^+,\dots,b_3^-\le 2$. 

We consider the the projection $q_{1,4}:K\to(\P^1)^2, (t_1,\dots,t_6)\mapsto (t_1,t_4)$. 
The image of this curve has bi-degree $(r_1,r_1)$, with $r_1\deg(q_{1,4})=b_1$.
The preimage of $(\pm\ci,\pm\ci)$ consists entirely of bonds; moreover, if one of the coordinates
of a point on $C_{1,4}$ is equal to $\pm\ci$, then it must already be a bond. If, say,
$b_1^+=b_2^+=2$, and $q_{1,4}$ is birational, then $+\ci$ is a branching point for both projections, 
hence it must be a double point. If $q_{1,4}$ is not birational, then it is a 2:1 map, because
the preimage of any of the 4 points $(\pm\ci,\pm\ci)$ is at most 2. In this case, the numbers
$b_1^+$ and $b_1^-$ are either 0 or 2, and the bi-degree of $C_{1,4}$ is $(1,1)$ or $(2,2)$.
It follows that, in the 2:1 case, the curve $C_{1,4}$ has genus 0 or 1.
We now have to sort out several cases.

Case 1: the three maps $q_{1,4}$, $q_{2,5}$ and $q_{3,6}$ are 2:1 maps. It is not possible that
all three maps factor through the same 2:1 quotient, because $K$ is contained in the product 
$C_{1,4}\times C_{2,5}\times C_{3,6}$. Assume, without loss of generality, that $q_{1,4}$ and $q_{2,5}$
do not factor by the same quotient 2:1 quotient. Then $(q_{1,4},q_{2,5}):K\to C_{1,4}\times C_{2,5}$ is
birational. By Lemma~\ref{lem:ag}, the image has genus at most five, and therefore $g(K)\le 5$.

For the remaining cases, we may assume that $q_{1,4}$ is birational.

Case 2: $b_1=3$. Then the arithmetic genus of $C_{1,4}$ is $(b_1-1)^2=4$. Since
$b_1^++b_1^-=3$, at least one of the two numbers is equal to two; assume, without loss of generality,
that $b_1^+=2$ and $b_1^-=1$. Then $(+\ci,+\ci)$ and $(-\ci,-\ci)$ are double points of $C_{1,4}$,
and therefore $g(K)\le 4-2=2$.

Case 3: $b_1\le 2$. Then the arithmetic genus of $C_{1,4}$ is $(b_1-1)^2\le 1$.

Case 4: $b_1=4$, hence $b_1^+=b_1^{-}=2$. Then the arithmetic genus of $C_{1,4}$ is $(b_1-1)^2=9$,
and all four points $(\pm\ci,\pm\ci)$ are double points. Then $g(C_{1,4})\le 9-4=5$, and
therefore $g(K)\le 5$.
\end{proof}

\begin{cor} \label{rem:8}
The maximal genus 5 is reached in Case~1 when all $C_{1,4}$, $C_{2,5}$, and $C_{3,6}$ are elliptic
and have bi-degree $(2,2)$, and in Case~3; in both cases, we have $b_1=b_2=b_3=4$.
\end{cor}

As a consequence of Lemma~\ref{lem:4}, Lemma~\ref{lem:6}, Lemma~\ref{lem:68}, and Lemma~\ref{lem:8}
above, we finally obtain our bound for the genus.

\begin{thm} \label{thm:bound}
The genus of the configuration curve of a closed 6R linkage is at most 5.
\end{thm}

By re-examining the proof of Lemma~\ref{lem:8} more closely, we can prove the following theorem
which will be useful later for classifying linkages with a genus 5 configuration curve.

\begin{thm} \label{thm:45}
If the bond diagram is different from the diagrams Figure~\ref{fig:hookediet}(a), (b), and (d),
then $g(K)\le 3$.
\end{thm}

\begin{proof}
In view of Remark~\ref{rem:4}, Lemma~\ref{lem:68}, and the proofs of Lemmas~\ref{lem:8} and \ref{lem:6}, 
we just need to consider the case where $l_{i,i+1,i+2}=8$ for $i=1,\dots,6$. Assume indirectly that $b_1<4$
(using the notation as in the proof of Lemma~\ref{lem:8}). If $q_{1,4}$ is birational, then it follows
$g(K)\le 3$, hence we may assume that $q_{1,4}$ is a 2:1 map. Hence $b_1=2$ and $C_{1,4}$ is a curve of
bi-degree $(1,1)$, which is rational. Consequently $K$ is hyperelliptic (or $g(K)\le 1$ and the proof
is finished).


If the other two maps $q_{2,5}$, $q_{3,6}$ are also 2:1 maps, then we have a 2:1 map from $K$ to a rational
curve and another 2:1 map to a curve of genus at most 1; by Lemma~\ref{lem:ag}, we obtain $g(K)\le 3$.

So we may assume there is another map, say $q_{2,5}:K\to C_{2,5}\subset(\P^1)^2$, which is birational.
Its image has then bi-degree $(b_2,b_2)$, and if $b_2\le 3$ then we again get $g(K)\le 3$. So we assume $b_2=4$.
Then $C_{2,5}$ has bi-degree $(4,4)$ and 4 double points $(\pm\ci,\pm\ci)$.
The canonical map of $C_{1,4}$ is defined by the polynomials of bi-degree $(2,2)$ passing to all 
$m$-fold singular points with order $m-1$. If there is at most one double point,
then it would just pass to the 4 double points $(\pm\ci,\pm\ci)$ and maybe one additional double point,
but this map maps $(\P^1)^2$ birational to a rational surface, and this contradicts to the fact that $C_{1,4}$
is  hyperelliptic, because the canonical map of a hyperelliptic curve is 2:1. Hence there must
be at least two more double points or a triple point on $C_{2,5}$, and so $g(K)\le 3$.
\end{proof}

\begin{rem} \label{rem:reduc}
If the configuration curve has more than one one-dimensional component, then one can define bonds
for the individual components. These bonds add up to a diagram which satisfies the same
conditions we just proved for bond diagrams of irreducible configuration curves. We conclude that
the genus of any component is at most 3 in a linkage with more than one component.
\end{rem}

\section{Quad Polynomials} \label{sec:quad}

In this section, we introduce a technique to derive algebraic equations on the parameters of a linkage
from the existence of bonds connecting opposite edges.

Let $L=(h_1,\dots,h_6)$ be a linkage, and let $\beta=(t_1,\dots,t_6)$ be a bond connecting $h_1$ and $h_4$.
We assume that $l_{1,2,3}=l_{6,5,4}=8$, and we fix $t_1$ and $t_4$ (e.g. $t_1=t_4=+\ci$). 
Let $G\subset\P^7$ be the line corresponding to the two-dimensional intersection 
of the left annihilator of $(t_4-h_4)$ and the right ideal $(t_1-h_1)\D\H_\C$ (see Lemma~\ref{lem:alg}).
The intersection of $G$ and $X_{1,2,3}$ can be computed by solving the vector-valued equation
$(\ci-h_1)(t_2-h_2)(t_3-h_3)(\ci-h_4)=0$ for $t_2,t_3$. Geometrically, this is the intersection
of a quadric surface with the linear subspace $\{x \mid (t_1-h_1)x(t_4-h_4)=0\}$. By Lemma~\ref{lem:alg},
this subspace has codimension~2. So the intersection is either $G$ or zero-dimensional of
degree~2. We can exclude the first case, because the lines on $X_{1,2,3}$ appear in three well-known families
(two of the three parameters being constant), and none of these families may contain $G$. Hence there
is a quadric univariate polynomial parametrizing the intersection of $G$ and $X_{1,2,3}$. 
Similarily, there is another univariate quadric polynomial parametrizing the intersection of $G$ and $X_{6,5,4}$.
The number of bonds connecting $h_1$ and $h_4$ is then bound above by the degree of the greatest common divisor
of these two polynomials. 

We describe this idea more concretely. Let $h_1,h_2,h_3,h_4$ be lines such that $l_{1,2,3}=8$.
We define the {\em quad polynomial} $Q_{h_1,h_2,h_3,h_4}\in\C[x]$ as the
unique normed generator of the elimination ideal $\C[x]\cap I$, where $I\subset \C[x,y,t_2,t_3]$
is the ideal generated by the coordinates of $(t_1-h_1)(t_2-h_2)(t_3-h_3)(t_4-h_4)$ and of
$(t_1-h_1)(t_2-h_2)(t_3-h_3)-y(1+x\eps)(+\ci-h_1)(+\ci+h_4)$. (Here we assume that
$h_1$ and $h_4$ are not parallel, which implies that $(+\ci-h_1)(+\ci+h_4)$ generates $G$ as a $\D$-module;
in the special case when $h_1$ and $h_4$ are parallel, we have to choose a different generator.)

In the following, we frequently write $a\equiv b$ for equality modulo multiplication by a nonzero
complex scalar (i.e. projective equality, or both sides equal to zero).

\begin{rem} \label{rem:linquad}
The degree of the quad polynomial is 2 unless there is a common intersection point of the coupling variety $X_{1,2,3}$,
the line $G$, and the linear 3-space consisting of all multiples of $\eps$. A closer analysis 
shows that the existence of such an intersection point is equivalent to either $h_1$ being parallel to $h_2$
or $h_3$ being parallel to $h_4$.
\end{rem}

\begin{thm} \label{thm:quad}
The number of bonds connecting $h_1$ and $h_4$ with $t_1=+\ci$ and $t_4=+\ci$, counted with multiplicity,
is less than or equal to the degree of the greatest common divisor of 
$Q_{h_1,h_2,h_3,h_4}$ and $Q_{h_4,h_5,h_6,h_1}$.

The number of bonds connecting $h_1$ and $h_4$ with $t_1=+\ci$ and $t_4=-\ci$, counted with multiplicity,
is less than equal to the degree of the greatest common divisor of 
$Q_{h_1,h_2,h_3,-h_4}$ and $Q_{-h_4,h_5,h_6,h_1}$.
\end{thm}

\begin{proof}
Let $\beta=(+\ci,t_2,t_3,+\ci,t_5,t_6)$ be a bond connecting $h_1$ and $h_4$. Then there exists a $z\in\C$
such that $f_{63}(\beta)\equiv (1+z\eps)(+\ci-h_1)(+\ci+h_4)$ in $\P^7$, and so
$Q_{h_1,h_2,h_3,h_4}(z)=0$. The image $f_{3,6}(\beta)$ in $X_{4,5,6}$ is the quaternion conjugate of
$f_{63}(\beta)$, which is equal to $(1+z\eps)(+\ci-h_4)(+\ci+h_1)$.
Hence $Q_{h_4,h_5,h_6,h_1}(z)=0$. So a bond gives
rise to a common root of $Q_{h_1,h_2,h_3,h_4}$ and $Q_{h_4,h_5,h_6,h_1}$. Also, a bond
with connection number two give rise to a common double root.

The second statement can be reduced to the first by replacing $h_4$ by its negative.
\end{proof}

\begin{rem} \label{rem:q1}
Of course there is an analogous statement for bonds with $t_1=-\ci$. However, we do not need these,
because these bonds are complex conjugate to bonds with $t_1=+\ci$. Indeed, it is straightforward to show
that $Q_{-h_1,h_2,h_3,-h_4}$ is the complex conjugate of $Q_{h_1,h_2,h_3,h_4}$ and 
$Q_{-h_1,h_2,h_3,h_4}$ is the complex conjugate of $Q_{h_1,h_2,h_3,-h_4}$. If we replace $h_2$ or $h_3$
by its negative, then the quad polynomial remains the same.
\end{rem}

\begin{rem} \label{rem:q2}
The argument of the proof of Theorem~\ref{thm:quad} can be partially reversed:
a common root of the quad polynomials $Q_{h_1,h_2,h_3,h_4}$ and $Q_{h_4,h_5,h_6,h_1}$ implies a 
common point of $X_{1,2,3}$ and $X_{6,5,4}$ with norm zero. Its preimage $\alpha\in (\P^1)^6$ satisfies
the equation $(t_1-h_1)\dots(t_6-h_6)=0$. But $\alpha$ is not necessarily a bond, because it also could be an
isolated intersection point of $X_{1,2,3}$ and $X_{6,5,4}$, and then it is not an element in the Zariski closure
of $K$.
\end{rem}

When the linkage moves, then the position of the lines change, and therefore also the quad polynomial
changes. In the following, we will replace the quad polynomials by similar polynomials which are invariant
under motions. They can be described in terms of geometric parameters which are also invariant under motion.
A well-known set of these invariant parameters are the Denavit--Hartenberg parameters \cite{husty10} defined as
follows. 

For $i=1,\dots,6$, $\phi_i$ is defined as the angle of the direction vectors of the directed lines 
$h_i$ and $h_{i+1}$. Since this angle is determined up to sign, we require $0\le\phi_i <\pi$.
We also set $c_i:=\cos(\phi)$. 

For $i=1,\dots,6$, $d_i$ is defined as the orthogonal distance of the lines $h_i$ and $h_{i+1}$.
The sign of $d_i$ is not well-defined (it would depend on an orientation of the common normal, which
we do not like to choose); we will discuss the ambiguity when it arises.

If the lines $h_i$, $h_{i+1}$ are not parallel, then define the Bennett ratios as  $b_i:=\frac{d_i}{\sin(\phi_i)}$. (We mean no conflict with the bond number introduced in the Proof of Lemma~8; actually, we will not
use these bond numbers from now on.)
The sign of these numbers is well-defined: if the scalar part of $h_ih_{i+1}$ is written as $u+v\eps\in\D$,
then $c_i=-u$ and $b_i=\frac{-v}{1-u^2}$.

If the lines $h_i$, $h_{i+1}$ are not parallel and $h_i$, $h_{i-1}$ are not parallel, then $s_i$ is defined
as the distance of the intersections of $h_i$ and the common normals of $h_i$ and $h_{i\pm 1}$ (this
parameters are called offsets). The sign of the offset is well-defined, because the two points lie on an oriented  line
induced by $h_i$.

It is well-known that the invariant parameters $c_1,\dots,c_6,b_1,\dots,b_6,s_1,\dots,s_6$ form a complete
system of invariants for all closed 6R linkages without adjacent parallel lines. In other words, if
two such linkages share all parameters, then there is a collection of rotations in the configuration 
of one of them that transform it into the second.
An extension to linkages with adjacent parallel lines is also well-known, but more technical. In this
paper, we will assume from now on that there are no parallel adjacent lines. (As a consequence, the
quad polynomials are always quadratic.)

Changing the orientation of $h_i$ has the following effect on the parameters: 
$c_i,b_i,c_{i-1},b_{i-1},s_i$ are multiplied by $-1$, and all other parameters stay the same.

\begin{rem}
The condition $l_{1,2,3}=4$ is equivalent to ($b_1=b_2=0$ and $s_2=0$), which is equivalent to the statement
that the lines $h_1,h_2,h_3$ meet in a common point. (If we had not exluded adjacent parallel lines,
then $c_1^2=c_2^2=1$ would be a second possibility.)

The condition $l_{1,2,3}=6$ is equivalent to ($b_1^2=b_2^2\ne 0$ and $s_2=0$) (see \cite{hegedus13b}, Theorem~1).
Indeed, ($b_1^2=b_2^2$ and $s_2=0$) is Bennett's characterization of three skew lines for the existence
of a fourth line forming a 4R linkage. This is the reason why the numbers $b_1,\dots,b_6$ are called ``Bennett ratios''.
\end{rem}

For moving a linkage, we take a configuration $(r_1,\dots,r_6)\in K$. We may choose one of the links
as a base which does not move. If $o_6$ is the base, then the transformed lines are
$(h_1',\dots,h_6')$, where 
\[ h_i'\equiv (r_1-h_1)\cdots(r_{i-1}-h_{i-1})h_i(r_{i-1}-h_{i-1})^{-1}\cdots(r_1-h_1)^{-1} \]
for $i=1,\dots,6$. Note that $h_1'=h_1$ and $h_6'=h_6$, because these two lines are attached to $o_6$.
The configuration set $K'$ of the transformed linkage is isomorphic to $K$ by the isomorphism
\[ (t_1,\dots,t_6) \mapsto (t_1',\dots,t_6') = 
	\left(\frac{1+r_1t_1}{r_1-t_1},\dots,\frac{1+r_6t_6}{r_6-t_6} \right) ; \]
this transformation has the property that $(t_i-h_i)(r_i+h_i)\equiv (t_i'-h_i)$
for $i=1,\dots,6$.

In order to study the effect of a motion on the quad polynomial, it is more convenient to choose
$o_2$ as the base. Then we have
\[ h_1'\equiv (r_2+h_2)h_1(r_2-h_2),\ h_2'=h_2,\ h_3'=h_3,\ h_4'\equiv (r_3-h_3)h_4(r_3+h_3) . \]

\begin{lem} \label{lem:shift}
The effect of motion on the quad polynomial is a variable shift: there exists a complex number $w$
such that $Q_{h_1',h_2',h_3',h_4'}(x)=Q_{h_1,h_2,h_3,h_4}(x+w)$. The shift is the same for the second
quad polynomial, i.e. we also have
$Q_{h_4',h_5',h_6',h_1'}(x)=Q_{h_4,h_5,h_6,h_1}(x+w)$.
\end{lem}

\begin{proof}
Let $G\subset\P^7$ be the line through $u:=(+\ci-h_1)(+\ci+h_4)$ and $\eps u$. The 
intersection of $G$ with $X_{1,2,3}$ is of the form $(1+z\eps)u$, 
where $z$ is a root of $Q_{h_1,h_2,h_3,h_4}$; let $(+\ci,t_2,t_3)$
be its parameter values. The motion transforms the parameters to
$(+\ci,t_2',t_3')=\left(+\ci,\frac{1+r_2t_2}{r_2-t_2},\frac{1+r_3t_3}{r_3-t_3} \right)$,
parametrizing a point on $X'_{1,2,3}$. Just for verification, we may compute
\[ (+\ci-h_1')(t_2'-h_2')(t_3'-h_3')(+\ci-h_4') \equiv \]
\[ (r_2+h_2)(+\ci-h_1)(r_2-h_2)(t_2'-h_2)(t_3'-h_3)(r_3-h_3)(+\ci-h_4)(r_3+h_3)\equiv \]
\[ (r_2+h_2)(+\ci-h_1)(t_2-h_2)(t_3-h_3)(+\ci-h_4)(r_3+h_3)= 0 . \]
Hence $p:=(+\ci-h_1')(t_2'-h_2')(t_3'-h_3')$ lies on the line $G'$ through
$u':=(+\ci-h_1')(+\ci+h_4')$ and $\eps u'$. Projectively, we can write $p\equiv (1+z'\eps)u'$,
where $z'$ is a root of the transformed quad polynomial $Q_{h'_4,h'_3,h'_2,h'_1}$. We multiply
both sides on the left with $(r_2-h_2)$ and from the right with $(r_3-h_3)$ and obtain
\[ (r_2-h_2)p(r_3-h_3)\equiv (+\ci-h_1)(t_2-h_2)(t_3-h_3) , \]
\[ (1+z'\eps)(r_2-h_2)u'(r_3-h_3) \equiv (1+z'\eps)(+\ci-h_1)(r_2-h_2)(r_3-h_3)(+\ci+h_4) . \]
Now we define $w$ as the unique complex number such that 
$(+\ci-h_1)(r_2-h_2)(r_3-h_3)(+\ci+h_4) \equiv (1-w\eps)u$. Then we get 
$1+z'\eps=(1+z\eps)(1-w\eps) = (1+(z-w)\eps)$, 
hence the shift by $w$ transforms the roots of $Q_{h_1,h_2,h_3,h_4}$ to the roots of 
$Q_{h_1',h_2',h_3',h_4'}$. Because both polynomials are normed, the first statement of the theorem is proved.

For the shift of the second quad polynomial, we also need
\[ h_5' \equiv (r_3-h_3)(r_4-h_4)h_5(r_4+h_4)(r_3+h_3), \]
\[ h_6' \equiv (r_2+h_2)(r_1+h_1)h_6(r_1-h_1)(r_2-h_2). \]
A similar computation (using also the closure equation) shows that the shift value is
again equal to $w$.
\end{proof}

Now, let $L:=(h_1,\dots,h_6)$ be a closed 6R linkage, not necessarily mobile, such that
$l_{i,i+1,i+2}=8$ for $i=1,\dots,6$. For all such $i$, we define $w_i$ as the arithmetic mean value
of all roots of $Q_{h_i,h_{i+1},h_{i+2},h_{i+3}}$ and $Q_{h_{i+3},h_{i+4},h_{i+5},h_{i}}$.
Note that $w_i=w_{i+3}$ for all $i$.
We define the {\em invariant quad polynomials} $Q_1^+,\dots,Q_6^+\in\C[x]$ by
\[ Q_i^+(x) := Q_{h_i,h_{i+1},h_{i+2},h_{i+3}}(x+w_i) . \] 
The polynomials $Q_1^-,\dots,Q_6^-\in\C[x]$ are defined as $Q_1^+,\dots,Q_6^+$ for the
linkage $L':=(-h_1,h_2,-h_3,h_4,-h_5,h_6)$.

Theorem~\ref{thm:quad} holds also for the invariant quad polynomials. But these polynomials  have the
additional advantage that they can be expressed in the invariant parameters.
Here is the precise statement.

\begin{thm} \label{thm:iquad}
Let $k\in\{1,\dots,6\}$.
The number of bonds connecting $h_k$ and $h_{k+3}$ with $t_k=+\ci$ and $t_{k+3}=+\ci$, 
counted with multiplicity,
is less than or equal to the degree of the greatest common divisor of $Q_k^+$ and $Q_{k+3}^+$.
Similarily, the number of bonds connecting $h_k$ and $h_{k+3}$ with $t_k=+\ci$ and $t_{k+3}=-\ci$, 
counted with multiplicity, 
is less than or equal to the degree of the greatest common divisor of $Q_k^-$ and $Q_{k+3}^-$.

The invariant quad polynomials are invariant under motion. More precisely, we have the formula
\[ Q_k^+(x) = \left(x+\frac{f_{k}-f_{k+2}-f_{k+3}+f_{k+5}}{2}+\ci \frac{s_{k+3}-s_k}{4}\right)^2 + \]
\[ \frac{c_{k+1}}{2}(b_kb_{k+2}-s_{k+1}s_{k+2})
	+\frac{\ci}{2}(s_{k+1}(b_k-b_{k+2}c_{k+1})+s_{k+2}(b_{k+2}-b_kc_{k+1})) , \]
where $f_k=b_kc_k$ for $k=1,\dots,6$.
The invariant quad polynomials $Q_k^-$ can be obtained by substituting $b_1,\dots,b_6,c_1,\dots,c_6$
by their negatives and $s_k$ by $(-1)^ks_k$ for $k=1,\dots,6$.
\end{thm}

\begin{proof}
By Theorem~\ref{thm:quad}, the number of bonds connecting $h_1$ and $h_4$ is less than or equal 
to the degree of the greatest common divisor of $Q_{h_1,h_2,h_3,h_4}$ and $Q_{h_4,h_5,h_6,h_1}$.
The invariant quad polynomials $Q_1^+$ and $Q_4^+$ are their shifts by $w_1=w_4$, hence their
greatest common divisor has the same degree. This implies the first assertion. The assertion
on $Q_1^-$ can be reduced to the first one.

By Lemma~\ref{lem:shift}, it is clear that the polynomials $Q_1^+$ and $Q_4^+$ are invariant up
to shift of the variable. In addition, the arithmetic mean of all their roots is $0$ by construction. 
This leaves no choice for the shift, which proves invariance.

Since the Denavit--Hartenberg parameters $b_k,c_k,s_k$, $k=1,\dots,6$, form a complete system of invariants, 
it follows that the coefficients of the invariant quad polynomials can be expressed as rational functions
in them. The actual formula has been calculated using the computer algebra system Maple, taking
advantage of the invariance.
\end{proof}

Using Theorem~\ref{thm:iquad}, we can formulate a necessary condition for the existence of a bond connecting $h_1$
and $h_4$: either the resultant of $Q_1^+$ and $Q_4^+$ or the resultant of $Q_1^-$ and $Q_4^-$ has to
vanish. Both resultant can be expressed as polynomials in $b_1,\dots,s_6$, but this polynomial turns
out to be relatively complicated. Fortunately one obtains an easier system of equations when the
maximal number of bonds is assumed.

\section{Linkages with Maximal Genus} \label{sec:max}

In this section we give a classification of all closed 6R linkages with a configuration curve of genus at least 4
that do not have links with parallel joint axes, in terms of their Denavit--Hartenberg parameters.
It turns out there are are 4 irreducible families; two of them are well-known, the other two are new.

As in the previous section, we use the angle cosines $c_1,\dots,c_6$, the Bennett ratios $b_1,\dots,b_6$
and the offsets $s_1,\dots,s_6$. In addition, we also use the f-values $f_k=c_kb_k$, $k=1,\dots,6$;
this leads to shorter formulas.

Let $L$ be a linkage such that no adjacent axes are parallel, and assume that the genus of its configuration curve
at least four. By Theorem~\ref{thm:45}, its bond diagram is Figure~\ref{fig:hookediet}(a), (b), or (d).
Cases (a) and (b) are well-known and have been described in the Lemmas~\ref{lem:4} and \ref{lem:6}:
these are the Hooke linkage and the Dietmaier linkage, respectively.

\begin{rem} \label{rem:hookediet}
Just for the sake of completeness, here is the description in terms of the Denavit--Hartenberg parameters
(see \cite{Dietmaier}).
\begin{description}
\item[Hooke linkage:] $b_1=b_2=b_4=b_5=s_2=s_5=0$, 
	$d_3^2+s_3^2+s_4^2-2c_3s_2s_4=d_6^2+s_2^2+s_5^2-2c_6s_1s_5$.
\item[Dietmaier linkage:] $b_1=b_2,b_4=b_5,b_3=b_6,c_3=c_6,f_1+f_2=f_4+f_5,s_1=s_3,s_4=s_6,s_2=s_5=0$
	up to orientation of the axes.
\end{description}
\end{rem}

From now on, we assume that $l_{k,k+1,k+2}=8$ for $k=1,\dots,6$; consequently, the bond diagram
is Figure~\ref{fig:hookediet}(d). The number of bonds is maximal, for $k=1,2,3$, and for any choice
of $t_k,t_{k+3}$ in $\{+\ci,-\ci\}$, there exist 2 bonds connecting $h_k$ and $h_{k+3}$.
By Theorem~\ref{thm:iquad}, we get the following equalities of polynomials in $\C[x]$:
\begin{equation} \label{eq:444}
  Q_1^+=Q_4^+, Q_2^+=Q_5^+, Q_3^+=Q_6^+, Q_1^-=Q_4^-, Q_2^-=Q_5^-, Q_3^-=Q_6^- . 
\end{equation}
Each equality of polynomials gives rise to 4 scalar equations, namely the real and imaginary part
of the linear and the constant coefficient. 

\begin{lem} \label{lem:sol}
The zero set of the 24 equations above is the union of two irreducible components. For both,
we have $s_1=\cdots=s_6=0$ and the three equations
\begin{equation} \label{eq:red}
 b_1c_2b_3=b_4c_5b_6, b_2c_3b_4=b_5c_6b_1, b_3c_4b_5=b_6c_1b_2 . 
\end{equation}
The two components are
\begin{enumerate}
\item $f_1=f_4, f_2=f_5, f_3=f_6, b_1b_3b_5=b_2b_4b_6$, \\ $b_1^2+b_3^2+b_5^2=b_2^2+b_4^2+b_6^2$
\item $f_1=f_3=f_5, f_2=f_4=f_6, b_1b_3b_5f_2=b_2b_4b_6f_1$, \\ $b_1^2+b_3^2+b_5^2+f_2^2=b_2^2+b_4^2+b_6^2+f_1^2$.
\end{enumerate}
If no Bennett ratio is zero, then the three equations (\ref{eq:red}) are redundant.
\end{lem}

\begin{proof}
By comparing the imaginary parts of the linear coefficients, it follows immediately that 
$s_1=\dots=s_6=0$. 
For the simplified system, we obtained the decomposition above by Gr\"obner basis computation
using the computer algebra system Maple.
\end{proof}

\begin{thm}
There are two irreducible families of 6R linkages with coupling dimensions~8 such that the
configuration curve has genus 5 generically.
They are characterized by (1) and (2) in Lemma~\ref{lem:sol}.
\end{thm}

\begin{proof}
The validity of the equations~(\ref{eq:444}) implies the existence of 24 points in the intersection
of $X_{1,2,3}$ and $X_{6,5,4}$, by Remark~\ref{rem:q2}. Intersection theory predicts an intersection of
only 16 points (see \cite[Section~11.5.1]{selig05}), therefore the intersection is infinite and the linkage moves.

Since the genus is a lower semicontinuous function in a family of curves, and 5 is the largest possible
value, it suffices to exhibit a single example with a configuration curve of genus 5 
for each of the two families in order to prove that 
the genus is 5 in the generic case. Here is an example that works for both, because it is
in the intersection of the two families:
\[ b_1=0, b_2=40, b_3=32, b_4=0, b_5=25, b_6=7, c_1=\dots=c_6=0. \]
\end{proof}


\begin{rem}
  A special case of the second family is Bricard's orthogonal linkage (see \cite{Baker80}). It can be
  characterized by the condition $s_1=\dots=s_6=c_1=\dots=c_6=0$ and
  $b_1^2+b_3^2+b_5^2=b_2^2+b_4^2+b_6^2$. The example in the proof of Theorem~\ref{thm:iquad} is
  actually an instance of Bricard's orthogonal linkage. Therefore we can conclude that the genus of
  the configuration curve of Bricard's orthogonal linkage is 5 generically.
\end{rem}

\begin{rem}
The linkages with a configuration curve of genus 4 are contained in the 4 families described in
this section as special cases. A concrete example is the Bricard orthogonal linkage with
$(b_1,\dots,b_6)=(4,3,5,7,9,8)$.
\end{rem}

\section*{Acknowledgements}

This research was supported by the Austrian Science Fund (FWF): DK~W~1214-N15.

\bibliographystyle{plain}
\bibliography{mega}

\end{document}